\documentclass[english, 11pt]{amsart}

\usepackage{tikz}
\usepackage{aliascnt}
\usepackage{mathrsfs}
\usepackage[all,poly,knot]{xy}
\usepackage{hyperref}
\usepackage{comment}  
\usepackage{csquotes}   
\usepackage{amssymb,amsbsy,amsmath,amsfonts,amssymb,amscd,
	graphics,color,footmisc,fancyhdr,multicol,fancybox,
	graphicx,mathrsfs,rotating,ifthen,wasysym}

\usepackage{times}
\usepackage{pdfpages}
\usepackage{multirow}
\usepackage{nccrules}
\usepackage{textcomp}
\usepackage{comment}
\usepackage{framed}
\usepackage[all]{xy}
\usepackage{graphicx}
\usepackage{pgf,tikz}
\usepackage{mathrsfs}
\usetikzlibrary{arrows}
\usepackage{lipsum}
\usepackage{mathtools}

\newcommand{\medcup}{\mathbin{\scalebox{1.5}{\ensuremath{\cup}}}}
\DeclareMathOperator{\gr}{Gr}
\DeclareMathOperator{\bF}{\mathbf{F}}
\DeclareMathOperator{\bK}{\mathbf{K}}

\DeclareMathOperator{\dif}{\text{\normalfont d}}
\DeclareMathOperator{\sym}{Sym}

\def\log{\mathrm{log}\,}

\theoremstyle{plain}

\newtheorem{thm}{Theorem}[section]

\newtheorem{lem}[thm]{{Lemma}}

\newtheorem{pro}[thm]{Proposition}

\newtheorem{defi}[thm]{Definition}

\theoremstyle{remark}
\newtheorem{rmk}[thm]{Remark}

\numberwithin{equation}{section}

\usepackage[top=1.5in, bottom=1.5in, left=1in, right=1in]{geometry}

 \usepackage[hyperpageref]{backref}
 
\usepackage{xcolor}
\hypersetup{
	colorlinks,
	linkcolor={red!50!black},
	citecolor={blue!62!black},
	urlcolor={blue!80!black}
}
\theoremstyle{plain}
\newcommand{\thistheoremname}{}
\newtheorem*{genericthm*}{\thistheoremname}
\newenvironment{namedthm*}[1]{\renewcommand{\thistheoremname}{#1}%
	\begin{genericthm*}}
	{\end{genericthm*}}

\newtheoremstyle{named}{}{}{\itshape}{}{\bfseries}{.}{.5em}{\thmnote{#3's }#1}
\theoremstyle{named}

\makeatletter
\newcommand\thankssymb[1]{\textsuperscript{\@fnsymbol{#1}}}
\makeatother
\begin{document} 
\title[Non-archimedean Ax-Lindemann Theorem]{\bf Big Picard Theorem for jet differentials\\ and Non-archimedean  Ax-Lindemann  Theorem} 



\author{Dinh Tuan Huynh}

\address{Hua Loo-Keng center for Mathematical Sciences, Academy of Mathematics and System Science, Chinese Academy of Sciences, Beijing 100190, China \& Department of Mathematics, University of Education, Hue University, 34 Le Loi St., Hue City, Vietnam}
\email{dinhtuanhuynh@hueuni.edu.vn}

\author{Ruiran Sun}
\address{Institut fur Mathematik, Universit\"at Mainz, Mainz, 55099, Germany}
\email{ruirasun@uni-mainz.de}

\author{Song-Yan Xie\thankssymb{1}}
\thanks{\thankssymb{1} \footnotesize{partially supported by  NSFC Grant No.~$11688101$}}
\address{Academy of Mathematics and System Science \& Hua Loo-Keng Key Laboratory
	of Mathematics, Chinese Academy of Sciences, Beijing 100190, China}
\email{xiesongyan@amss.ac.cn}

\begin{abstract}
	By implementing jet differential techniques in non-archimedean geometry, we obtain a big Picard type extension theorem, which generalizes a previous result of Cherry and Ru. As applications, we establish two hyperbolicity-related results. Firstly, we prove a non-archimedean Ax-Lindemann theorem for totally degenerate abelian varieties. Secondly, we show the pseudo-Borel hyperbolicity for subvarieties of general type in abelian varieties.
\end{abstract}

\maketitle

\section{\bf Introduction }

The classical Lindemann-Weierstrass Theorem states that, if $\alpha_1,\dots,\alpha_n$ are $\mathbb{Q}$-linearly independent algebraic numbers, then their exponentials $\mathrm{exp}(\alpha_1),\dots,\mathrm{exp}(\alpha_n)$ are $\mathbb{Q}$-algebraically independent. 
To understand a conjecture of Schanuel, Ax established the differential field analogues of it in \cite{Ax71}, which can be formulated in the following geometrical way 
\begin{namedthm*}{Ax-Lindemann Theorem}
	(i) (for tori)
Let $\pi_{\mathrm{t}}:= (\mathrm{exp}(2\pi i \cdot),\dots,\mathrm{exp}(2\pi i \cdot)):\,\mathbb{C}^n \to (\mathbb{C}^*)^n$ be the uniformizing map and let $Y \subset \mathbb{C}^n$ be an algebraic subvariety. Then any irreducible component of the Zariski closure of the image $\pi_{\mathrm{t}}(Y)$ is a translate of some subtorus of $(\mathbb{C}^*)^n$.\\
(ii) (for abelian varieties)
Let $\pi_{\mathrm{a}}:\,\mathbb{C}^n \to A$ be the uniformizing map of a complex abelian variety $A$ and let $Y \subset \mathbb{C}^n$ be an algebraic subvariety. Then any irreducible component of the Zariski closure of the image $\pi_{\mathrm{a}}(Y)$ is a translate of some abelian subvariety of $A$.
\end{namedthm*} 
{The uniformizing map $\pi_{\mathrm{t}}$ (resp. $\pi_{\mathrm{a}}$) satisfies the \emph{exponential differential equation} associated to the commutative group scheme $(\mathbb{G}_m)^n$ (resp. $A$) over $\mathbb{C}$. For an arbitrary differential field $\bK$ of characteristic zero and for any semiabelian variety $G$ over $\bK$, one can still formulate the exponential differential equation associated to $G$ and $\mathrm{Lie}\,G$, albeit there is in general no uniformization of $G$ by $\mathrm{Lie}\,G$. In \cite{kir} Kirby proved an analogue of Ax's theorem in this setting, which over $\bK=\mathbb{C}$ implies the above statement.} 

Along another direction, Pila in his seminar paper \cite{Pila} obtained a hyperbolic version of the above Ax-Lindemann theorem, which concerns about the (orbifold-) uniformization of products of modular curves.
This version of Ax-Lindemann theorem plays an essential role in his proof of Andr\'{e}-Oort conjecture for products of modular curves.
Thereafter, Ullmo, Yafaev and Klingler
   established further  generalization  in a broader setting \cite{Ull14, UY14, KUY16, KUY18}, see also \cite{Peterzil-Starchenko17, Dinh-Viet2020} for other generalizations. Such transcendence results have drawn extensive attention, by its own interests as well as the linkage to prominent problems, e.g., the Andr\'{e}-Oort conjecture \cite{PT14, Tsi18}.
   
  Parallelly,
it would be natural to seek Ax-Lindemann type results  in \emph{non-archimedean} settings. However, most complex uniformizations are not known to have appropriate non-archimedean counterparts at the moment. For instance, every complex abelian variety $\mathbb{C}^n/\Gamma$ of dimension $n$ can be uniformized by $ \mathbb{C}^n\rightarrow \mathbb{C}^n/\Gamma$ which is an infinite, transcendental  morphism, while an abelian variety over a non-archimedean field with \emph{good reduction} admits no such uniformization.

Accessibly, for products of hyperbolic Mumford curves which are known to have nice non-archimedean uniformizations by Mumford's theory \cite{Mum72}, Chambert-Loir and Loeser~\cite{CL17} established a non-archimedean analogue of Pila's hyperbolic Ax-Lindemann theorem. 
Inspired by their work, in this paper we prove 

\begin{thm}[Non-archimedean Ax-Lindemann]
		\label{thm-A}
Let $\bF$ be a complete algebraically closed non-archimedean valued field of characteristic zero.
Let $p:\, (\mathbb{G}^{\mathrm{an}}_{m,\bF})^g \to A^{\mathrm{an}}$ be the uniformizing map of a totally degenerate abelian variety $A$ over $\bF$. Let $X \subset (\mathbb{G}_{m,\bF})^g$ be an affine variety over $\bF$. Then any irreducible component of the Zariski closure of the image $p(X^{\mathrm{an}})$ is a translate of some abelian subvariety of $A$.
\end{thm}
{We emphasize here the difference between non-archimedean uniformization and complex uniformization. 
In our situation the universal covering space is \emph{not} the Lie algebra $\mathrm{Lie}\,A$ but the \emph{rigid torus} $(\mathbb{G}_{m,\bF})^g$ which comes from the identity component of the N\'{e}ron model of the totally degenerate abelian variety $A$, and the uniformizing map $p:\, (\mathbb{G}^{\mathrm{an}}_{m,\bF})^g \to A^{\mathrm{an}}$ is constructed from formal lifting of torus and applying the functor of Raynaud's generic fiber (cf. Subsection~\ref{non-arc_unif} for details). When $g=1$ this is exactly the famous Tate curve. Due to the abstract nature of the non-archimedean uniformization, it is in general not clear that which differential equation should be satisfied by the uniformizing map $p$, and thus we do not know how to use Kirby's differential field argument to derive Theorem~\ref{thm-A}.}

{
	The non-archimedean uniformization is a natural analogue of the complex one, even though its relation with Lie algebra and differential equation is unclear. It plays crucial role in many important results. In the fundamental work \cite{Che94}, by employing the aforementioned non-archimedean universal coverings of abelian varieties, Cherry proved that over a non-archimedean valued field, any analytic map from $\mathbb{A}^1$ or $\mathbb{G}_m$ to some abelian variety must be constant. In \cite{Liu11}, non-archimedean uniformization of totally degenerate abelian varieties is used for understanding a non-archimedean analogue of the Calabi-Yau theorem. In \cite{Morrow2020}, the non-archimedean hyperbolicity of subvarieties of abelian varieties is proved, and in the  proof, the non-archimedean uniformization plays again a significant role.}  

We briefly explain the strategy to prove Theorem~\ref{thm-A}.
Avoiding the difficulty in handling the uniformizing differential equation in non-archimedean case, we use a geometrical approach coming from complex hyperbolic geometry. More precisely, our proof is an adaptation of a strategy due to Noguchi  \cite{Nog18} using Nevanlinna theory. Indeed, we will implement the jet differential method in the non-archimedean setting, by establishing the following Schwartz's Lemma type result for higher order jet differentials,
which generalizes a theorem of Cherry and Ru~\cite{Cherry-Ru2004} about $1$-jet.

\begin{thm}[Big Picard Theorem for jet differentials]
	\label{big-Picard}
	Let $X$ be a nonsingular projective variety defined over  $\bF$ and let $D$ be a simple normal crossing divisor on $X$. Let $f:A(0,R]\rightarrow X\setminus D$ be a rigid analytic map. If there exists some global logarithmic jet differential form along $D$ vanishing on some ample line bundle $\mathcal{A}$ on $X$, namely 
	\[
	\omega\in H^0\big(X,E^{GG}_{k,m}T_X^*(\log D)\otimes\mathcal{A}^{-1}\big),
	\] 
	such that $f^*\omega\not\equiv 0$, then $f$ extends to a rigid analytic map from $A[0,R]$ to $X$.
\end{thm} 


As a second application of the above theorem, 
we prove the following result
concerning 
\emph{Borel hyperbolicity}, a notion introduced by 
Javanpeykar and Kucharczyk \cite{JK20} to capture the essence of Borel's Algebraicity Theorem \cite{Bor72} for arithmetic varieties.

\begin{thm}[Non-archimedean pseudo-Borel hyerbolicity]
	
	\label{thm-B}
Suppose that $X$ is a closed subvariety of general type contained in an abelian variety $A$ over $\bF$. Denote by $\mathrm{Sp}(X)$ the union of translates of positive-dimensional abelian subvarieties of $A$ contained in $X$.
Then $X$ is $\bF$-analytically Borel hyperbolic modulo $\mathrm{Sp}(X)$. That is, for any algebraic variety $S$ over $\bF$, any rigid analytic map $f:\,S^{\mathrm{an}} \to X^{\mathrm{an}}$ with $f(S^{\mathrm{an}}) \not\subset \mathrm{Sp}(X)^{\mathrm{an}}$ is induced from an algebraic morphism.  
\end{thm}

Here is the outline of this paper. In Section~\ref{section 2} we recall some basic non-archimedean Nevanlinna theory. Next, in Section~\ref{Section 3}
we present the jet differential technique,
and establish an analytic extension type Theorem~\ref{big-Picard}, which is a key ingredient of our strategy. Lastly, in Section~\ref{Section 4} we prove our main Theorems~\ref{thm-A},~\ref{thm-B} 
in the same vein as~\cite{Nog18}, with  extra effort of 
 topological arguments to overcome the difficulty  caused by the absence of ``area-comparison method'' in non-archimedean geometry, see  Subsection~\ref{area argument}.
 
 \bigskip
\noindent{\bf Acknowledgment.} This paper is inspired by Professor Noguchi's talk given in the workshop ``Diophantine Approximation and Value Distribution Theory'' in Montreal, 2019. We thank D.-V. Vu for his interest and comments on our paper.
We also thank A. Javanpeykar, Professor S.S.Y Lu, Professor J. Xie and Professor Yamanoi for their interests and encouragements.
This paper was written during a visit of R. Sun to the Academy of Mathematics and Systems Science (AMSS) in Beijing, and he would like to thank AMSS and its members for their hospitality during the preparation of this paper. D. T. Huynh and S.-Y. Xie are grateful to AMSS for nice working conditions. Huynh also acknowledges the partial support of the Core Research Program of Hue University, Grant No. NCM.DHH.2020.15. R. Sun would like to thank Professor K. Zuo for his constant supports and encouragements.


\section{\bf Non-archimedean Nevanlinna theory}
\label{section 2}

We give a brief introduction about non-archimedean Nevanlinna theory and  collect some useful results for our applications. The reader is referred to~\cite{Cherry-Julie2002,Cherry-Ru2004} for more details. Let $\bF$ be an algebraically closed field, with characteristic zero, being complete with respect to a non-archimedian valuation $| \cdot |$. For $0\leq r_1<r_2\leq+\infty$, let $A[r_1,r_2]$ denote the closed annulus $\{z\in\bF:r_1\leq |z|\leq r_2\}$. Analogously, we denote the semi-bordered annuli by $A[r_1,r_2):=\{z\in\bF:r_1\leq |z|<r_2\}$ with the convention that $A[r_1,\infty):=\{z\in\bF:|z|\geq r_1\}$, and by $A(r_1,r_2]:=\{z\in\bF:r_1< |z|\leq r_2\}$. 
Set $|\bF|:=\{|z|:z\in\bF\}$. Then for any $r_1,r_2\in|\bF|$, the annulus $A[r_1,r_2]$ is affinoid. 

An {\sl analytic function} $f$ on the annulus $A[r_1,r_2]$ is a Laurent series 
\[
f(z)=\sum_{n\in\mathbb{Z}}a_nz^n,
\]
where the coefficients $a_n\in\bF$ satisfy $
\lim_{|n|\rightarrow\infty}|a_n|r^n=0$ for any $r_1\leq r\leq r_2$.
This definition can be generalized accordingly to any  type of annulus mentioned above. 

For each $r\in [r_1,r_2]$, we define
\[
|f|_r:=\sup_n|a_n|r^n,
\]
which turns out to be a non-archimedean absolute valuation on the ring of analytic functions on $A[r_1,r_2]$. A function is said to be {\sl meromorphic} on $A[r_1,r_2]$ if it can be written as a quotient of two analytic function on this annulus. By multiplicity law, the absolute valuation $|\cdot |_r$ can be extended to meromorphic functions.

Following \cite{Cherry-Ru2004}, an analytic function $f$ on $A[r_1,\infty)$ is said to be {\it analytic at infinity} if $f(z^{-1})$ can be extended to an analytic function on $A[0, \frac{1}{r_1}]$. A meromorphic function $f$ on $A[r_1,\infty)$ is said to be {\it meromorphic at infinity}
if $f(z^{-1})$ is meromorphic on $A[0, \frac{1}{r_1}]$.

Now we introduce the standard notations of Nevanlinna theory in the non-archimedean setting. Let $f=\sum_{n\in\mathbb{Z}}a_nz^n$ be an analytic function on the annulus $A[r_1,r_2]$. The {\sl proximity function} of $f$ at a point $a\in\bF$ is defined by
\[
m_f(a,r):=\max \{0,
-\log |f-a|_r\}\eqno\scriptstyle{(r_1\,\leq\,r\,\leq\,r_2)}.
\]
At the infinity where $a=\infty$, we set $m_f(r):=m_f(\infty,r):=\max\{0,\log|f|_r\}$.

The definition of a suitable {\sl counting function} for $f$ is delicate. We first denote
\[
k_f(r):=\inf\{n\in\mathbb{Z}:|f|_r=|a_n|r^n\}\quad\text{and}\quad K_f(r):=\sup\{n\in\mathbb{Z}:|f|_r=|a_n|r^n\}\eqno\scriptstyle{(r_1\,\leq\,r\,\leq\,r_2)},
\]
with the convention that $k_f(0)=0$ and $K_f(0)=\inf\{n:a_n\not=0\}$ when $r=0$ and $f(0)=0$. Then by a non-archimedean analogue of the Weierstrass Preparation Theorem, we can show that, for any $r$ and $R$ with $r_1\leq r\leq R\leq r_2$, the analytic function $f$ has exactly $K_f(R)-k_f(r)$ zero, counting multiplicity, in the annulus $A[r,R]$ (cf.~\cite{Cherry-Julie2002}). Hence it is  natural to define the {\sl counting functions} of $f$ by
\[
n_f(0,r):=K_f(r)-k_f(r)\quad\text{and}\quad N_f(0,r):=\int_{r_1}^rn_f(0,t)\dfrac{\dif t}{t}\eqno\scriptstyle{(r_1\,\leq\,r\,\leq\,r_2)}.
\]
In the case where $r_1=0$, the above definition is modified to be
\[
N_f(0,r)=n_f(0,0)
+
\int_0^r[n_f(0,t)-n_f(0,0)]\dfrac{\dif t}{t}\eqno\scriptstyle{(0\,\leq\,r\,\leq\,r_2)}.
\]
For a point $a\in \bF$, the counting function of $f$ with respect to $a$ is defined by $N_f(a,r)=N_{f-a}(0,r)$. The above definitions of proximity function and counting function can be generalized to the case where $f$ is meromorphic, see \cite{Cherry-Ru2004} for more details.
Regarding a meromorphic function $f$  as an analytic function on $\mathbb{P}^1(\bF)$, we can define $N_f(r,a)$ for any meromorphic function $f$ and any point $a\in\mathbb{P}^1(\bF)$.

Lastly, the {\sl characteristic function} of $f$ is given by
\[
T_f(r)=m_f(\infty,r)+N_f(\infty,r) \eqno\scriptstyle{(r_1\,\leq\,r\,\leq\,r_2)}.
\]

As an application of  the Poisson--Jensen Formula in the non-archimedean setting (cf. \cite{Cherry-Julie2002}), one receives
\begin{namedthm*}{First Main Theorem}
	Let $f$ be a meromorphic function on the annulus $A[r_1,r_2)$. Then for any point $a$ in $\mathbb{P}^1(\bF)$, there holds
	\[
	T_f(r)=m_f(a,r)+N_f(a,r)+O(\log r).
	\]	
	Furthermore, if $r_1=0$, then one can replace the error term $O(\log r)$ in the above estimate by a bounded term $O(1)$.
\end{namedthm*}

Most of the results in the classical Nevanlinna Theory have their counterparts in our current setting. We recall some of them below, which will be used in the sequence (see~\cite{Cherry-Ru2004}).

\begin{namedthm*}{Logarithmic Derivative Lemma} Let $f$ be a nonconstant meromorphic function on the annulus $A[r_1,r_2)$. For each $r\in[r_1,r_2)$ and for each integer $k\in\mathbb{N}$, one has
	\[
	\left|\dfrac{\dif^kf}{f}\right|_r
	\leq
	\dfrac{1}{r^k}
	\]
	and
	\[
	\left|\dif^k\log f\right|_r
	\leq
	\dfrac{C}{r^k},
	\]
	for some constant $C$ independent of $f$.
\end{namedthm*}

In order to prove the Theorem~\ref{big-Picard} we also need the following two results of Cherry and Ru~\cite{Cherry-Ru2004}.

\begin{pro}
	\label{existence of very ample divisor meeting f often}
	Let $X$ be a nonsingular projective variety and let $L$ be a very ample divisor on $X$. Then for any nonconstant rigid analytic map $f:A[r_1,r_2)\rightarrow X$ and for any $\epsilon>0$, one can find some very ample divisor $H$ in the complete linear system of $L$ such that
	\[
	N_f(H,r)\geq (1-\epsilon)T_f(H,r)+O(\log r).
	\]
\end{pro}

\begin{pro}
	\label{extension condition}
	Let $X$ be a nonsingular projective variety with a very ample divisor $H$ on it. Let $f$ be a rigid analytic map from the annulus $A(0,R]$ to $X$. If the rigid analytic map $g:A[1/R,\infty)$ obtained from $f$ by setting $g(\xi)=f(1/\xi)$ satisfies
	\[
	T_g(H,r)=O(\log r),
	\]
	then $f$ extends to a rigid analytic map $\widetilde{f}:A[0,R]\rightarrow X$.
\end{pro}

\section{\bf  Jet differentials and non-archimedean Big Picard Theorem}
\label{Section 3}

\subsection{Non-archimedean logarithmic Green-Griffiths $k$-jet bundles}

Let $\bF$ be an algebraically closed field, with characteristic zero, being complete with respect to a non-archimedian valuation $| \cdot |$. Let $X$ be a projective variety defined over $\bF$. For a point $x\in X$, consider the analytic germs $(\bF,0)\rightarrow (X,x)$. Two such germs are said to be equivalent if they have the same Taylor expansion up to order $k$ in some local coordinates around $x$. The equivalence class of an analytic germ $f\colon (\bF,0)\rightarrow (X,x)$ is called the {\sl $k$-jet of $f$}, denoted by $j_k(f)$, which is independent of the choice of local coordinates. A $k$-jet $j_k(f)$ is said to be {\sl regular} if $\dif f(0)\not=0$. For a given point $x\in X$, denote by $j_k(X)_x$ the vector space of all $k$-jets of analytic germs $(\bF,0)\rightarrow (X,x)$, set
\[
J_k(X)
:=
\underset{x\in X}{\medcup}\,J_k(X)_x,
\]
and consider the natural projection
\[
\pi_k\colon J_k(X)\rightarrow X.
\]
Then $J_k(X)$ carries the structure of an analytic fiber bundle over $X$, which is called the {\sl $k$-jet bundle over $X$}. Note that in general, $J_k(X)$ is not a vector bundle. When $k=1$, the $1$-jet bundle $J_1(X)$ is canonically isomorphic to the tangent bundle $T_X$ of $X$.

For an open subset $U\subset X$, for a section $\omega\in H^0(U,T_X^*)$, for a $k$-jet $j_k(f)\in J_k(X)|_U$, the pullback $f^*\omega$ is of the form $A(\xi)\dif\xi$ for some analytic function $A$, where $\xi$ is the global coordinate of $\bF$. Since each derivative $A^{(j)}$ ($0\leq j\leq k-1$) is well-defined, independent of the representation of $f$ in the class $j_k(f)$, the analytic $1$-form $\omega$ induces the analytic map
\begin{equation}
\label{trivialization-jet}
\tilde{\omega}
\colon
J_k(X)|_U
\rightarrow
\bF^k;\,\,j_k(f)\rightarrow
\big(A(z),A(z)^{(1)},\dots,A(z)^{(k-1)}
\big).
\end{equation}
Hence on an open subset $U$, a given local analytic coframe $\omega_1\wedge\dots\wedge\omega_n\not=0$ yields a trivialization 
\[
H^0(U, J_k(X))\rightarrow U\times(\bF^k)^n
\]
 by providing the following new $n k$ independent coordinates:
\[
\sigma\rightarrow(\pi_k\circ\sigma;\tilde{\omega}_1\circ\sigma,\dots,\tilde{\omega}_n\circ\sigma),
\]
where $\tilde{\omega}_i$ are defined as in \eqref{trivialization-jet}. The components $x_i^{(j)}$ ($1\leq i\leq n$, $1\leq j\leq k$) of $\tilde{\omega}_i\circ\sigma$ are called the {\sl jet-coordinates}.
In a more general situation where $\omega$ is a section over $U$ of the sheaf of meromorphic $1$-forms, the induced map $\tilde{\omega}$ is meromorphic.

Now, in the logarithmic setting, let $D\subset X$ be a normal crossing divisor on $X$. This means that at each point $x\in X$, there exist some local coordinates $z_1,\dots,z_{\ell},z_{\ell+1},\dots,z_n$ ($\ell=\ell(x)$) centered at $x$ in which $D$ is defined by
\[
D=
\{z_1\dots z_{\ell}=0
\}.
\]
Following Iitaka \cite{Iitaka1982}, the {\sl logarithmic cotangent bundle of $X$ along $D$}, denoted by $T_X^*(\log D)$, corresponds to the locally free sheaf generated by
\[
\dfrac{\dif\!z_1}{z_1},\dots,\dfrac{\dif\! z_{\ell}}{z_{\ell}},z_{\ell +1},\dots,z_n
\]
in the above local coordinates around $x$.

An analytic section $s\in H^0(U,J_k(X))$ over an open subset $U\subset X$ is said to be a {\sl logarithmic $k$-jet field} if $\tilde{\omega}\circ s$ are analytic for all sections $\omega\in H^0(U',T_X^*(\log D))$, for all open subsets $U'\subset U$, where $\tilde{\omega}$ are induced maps defined as in \eqref{trivialization-jet}. Such logarithmic $k$-jet fields define a
subsheaf of $J_k(X)$, and this subsheaf is itself a sheaf of sections of a analytic fiber bundle over $X$, called the {\sl logarithmic $k$-jet bundle over $X$ along $D$}, denoted by $J_k(X,-\log D)$ (see \cite{Noguchi1986}).

The group $\bF^*$ admits a natural fiberwise action defined as follows. For local coordinates 
\[
z_1,\dots,z_{\ell},z_{\ell+1},\dots,z_n\eqno\scriptstyle{(\ell=\ell(x))}
\]
centered at $x$ in which $D=\{z_1\dots z_{\ell}=0\}$, for any logarithmic $k$-jet field along $D$ represented by some germ $f=(f_1,\dots,f_n)$, if $\varphi_{\lambda}(z)=\lambda z$ is the homothety with ratio $\lambda\in \bF^*$, the action is given by
\[
\begin{cases}
\big(\log(f_i\circ \varphi_{\lambda})\big)^{(j)}
=
\lambda^j
\big(\log f_i\big)^{(j)}\circ\varphi_{\lambda} &
\quad
\scriptstyle{(1\,\leq\,i\,\leq\,\ell),}
\\
\big(f_i\circ \varphi_{\lambda}\big)^{(j)}
\quad\quad\,=
\lambda^j f_i^{(j)}\circ\varphi_{\lambda}
&
\quad
\scriptstyle{(\ell+1\,\leq\,i\,\leq\,n).}
\end{cases}
\]

Now we are in position to introduce the {\sl logarithmic Green-Griffiths $k$-jet bundle} \cite{Green-Griffiths1980} in non-archimedean setting. By a {\sl logarithmic jet differential of {\sl order} $k$ and {\sl degree}} $m$ at a point $x\in X$, we mean a polynomial $Q(f^{(1)},\dots,f^{(k)})$ on the fiber over $x$ of $J_k(X,-\log D)$ enjoying weighted homogeneity:
\[
Q(j_k(f\circ\varphi_{\lambda}))
=
\lambda^m
Q(j_k(f))
\eqno\scriptstyle{(\lambda\,\in\,\bF^*)}.
\]

Consider the symbols
\[
\dif^{j}\log z_i\eqno\scriptstyle{(1\,\leq\, j\,\leq\, k,\,1\,\leq\, i\,\leq\,\ell)}
\]
and
\[
\dif^{j}z_i\eqno\scriptstyle{(1\,\leq\, j\,\leq\, k,\,\ell\,+\,1\,\leq\, i\,\leq\,n)}.
\]
Set the weight of $\dif^{j}\log z_i$ or $\dif^{j}z_i$ to be $j$. Then
a logarithmic jet differential of order $k$ and weight $k$ along $D$ at $x$ is a weighted homogeneous polynomial of degree $m$ whose variables are these symbols. Denote by $E_{k,m}^{GG}T_X^*(\log D)_x$ be the vector space spanned by such polynomials and set
\[
E_{k,m}^{GG}T_X^*(\log D)
:=
\underset{x\in X}{\medcup}\,
E_{k,m}^{GG}T_X^*(\log D)_x.
\]
By Fa\`{a} di bruno's formula \cite{Constantine1996,Merker2015}, one can check that $E_{k,m}^{GG}T_X^*(\log D)$ carries the structure of a vector bundle over $X$, called {\sl logarithmic Green-Griffiths vector bundle}. A global section of $E_{k,m}^{GG}T_X^*(\log D)$ is called a {\sl logarithmic jet differential} of order $k$ and weight $m$ along $D$. Locally, a logarithmic jet differential form can be written as

{\footnotesize
	\begin{equation}
	\label{local expression of log jet}
	\underset{|\alpha_1|+2|\alpha_2|+\dots+k|\alpha_k|=m}
	{\sum_{\alpha_1,\dots,\alpha_k\in\mathbb{N}^n}}
	A_{\alpha_1,\dots,\alpha_k}
	\bigg(
	\prod_{i=1}^{\ell}
	\big(
	\dif\log z_i
	\big)^{\alpha_{1,i}}
	\prod_{i=\ell+1}^{n}
	\big(
	\dif z_i
	\big)
	^{\alpha_{1,i}}
	\bigg)
	\dots
	\bigg(
	\prod_{i=1}^{\ell}
	\big(
	\dif^k\log z_i
	\big)^{\alpha_{k,i}}
	\prod_{i=\ell+1}^{n}
	\big(
	\dif^kz_i
	\big)
	^{\alpha_{k,i}}
	\bigg),
	\end{equation}
}
where
\[
\alpha_{\lambda}
=
(\alpha_{\lambda,1},\dots,\alpha_{\lambda,n})
\in\mathbb{N}^n
\eqno
\scriptstyle{(1\,\leq\,\lambda\,\leq\, k)}
\]
are multi-indices of length 
\[
|\alpha_{\lambda}|
=
\sum_{1\leq i\leq n}
\alpha_{\lambda,i},
\]
and where $A_{\alpha_1,\dots,\alpha_k}$ are locally defined analytic functions.  As in the complex case, $E_{k,m}^{GG}T_X^*(\log D)$ admits the following natural filtration 

\begin{equation}
\label{filtration of Ekm}
\gr^{\bullet}E_{k,m}^{GG}T_X^*(\log D)
=
\underset{\ell_1+2\,\ell_2+\dots+k\,\ell_k\,=\,m}
{\bigoplus_{\ell_1,\,\ell_2,\,\dots,\,\ell_k\,\in\,\mathbb{N}}}
\sym^{\ell_1}T_X^*(\log D)
\otimes
\dots
\otimes
\sym^{\ell_k}T_X^*(\log D),
\end{equation}
where $\sym^{\ell_i}T_X^*(\log D)$ denote the symmetric powers of the logarithmic cotangent bundle.

We first obtain the following generalization of \cite[Proposition 4.2]{Cherry-Ru2004} for symmetric $1$-forms.

\begin{pro}
	\label{affinod covering and good express of jet}	
	Let $X$ be a nonsingular projective variety defined over $\bF$. Let $D$ be a simple normal crossing divisor on $X$. Then there exists a finite affinoid covering $\mathcal{U}$ of $X$ such that for any logarithmic jet differential form $\omega\in H^0(X,E_{k,m}^{GG}T_X^*(\log D)$, there exists a finite collection $\mathcal{R}$ of rational functions on $X$ such that on each $U\in\mathcal{U}$, one can write
	{\footnotesize
		\begin{equation}
		\label{local expression of jet differential via rational functions}
		\omega
		=
		\underset{|\alpha_1|+2|\alpha_2|+\dots+k|\alpha_k|=m}
		{\sum_{\alpha_1,\dots,\alpha_k\in\mathbb{N}^n}}
		\psi_{\alpha_1,\dots,\alpha_k}
		\bigg(
		\prod_{i=1}^{\ell_U}
		\big(
		\dif\log \phi_i
		\big)^{\alpha_{1,i}}
		\prod_{i=\ell_U+1}^{n}
		\bigg(
		\dfrac{\dif \phi_i}{\phi_i}
		\bigg)
		^{\alpha_{1,i}}
		\bigg)
		\dots
		\bigg(
		\prod_{i=1}^{\ell_U}
		\big(
		\dif^k\log \phi_i
		\big)^{\alpha_{k,i}}
		\prod_{i=\ell_U+1}^{n}
		\bigg(
		\dfrac{\dif^k\phi_i}{\phi_i}
		\bigg)
		^{\alpha_{k,i}}
		\bigg),
		\end{equation}
	}
	for some  $\phi_i\in\mathcal{R}$ and $\psi_{\alpha_1,\dots,\alpha_k}$ are rational functions which are all regular on $U$.
\end{pro}

\begin{proof}
	By compactness argument, one can take a finite covering $\mathcal{U}$ of $X$ by affinoid subdomain of $X$ such that for any $U\in\mathcal{U}$, the local defining equation of $D$ is given by $D=\{z_{U,1}\cdots z_{U,\ell_U}=0\}$, where $z_{U,1},\dots z_{U,\ell_U}$ are parts of a local coordinates system on $U$. For each $U\in\mathcal{U}$, let $V_U$ be the associated affine open set on $X$. By~\eqref{local expression of log jet}, there are rational functions $\phi_i,\chi_{\alpha_1,\dots,\alpha_k}\in \mathcal{O}_X(V)$ such that
	{\footnotesize
		\begin{align}
		\label{local rational expression of log jet}
		\omega|_V
		&=
		\underset{|\alpha_1|+2|\alpha_2|+\dots+k|\alpha_k|=m}
		{\sum_{\alpha_1,\dots,\alpha_k\in\mathbb{N}^n}}
		\chi_{\alpha_1,\dots,\alpha_k}
		\bigg(
		\prod_{i=1}^{\ell_U}
		\big(
		\dif\log \phi_i
		\big)^{\alpha_{1,i}}
		\prod_{i=\ell_U+1}^{n}
		\big(
		\dif \phi_i
		\big)
		^{\alpha_{1,i}}
		\bigg)
		\dots
		\bigg(
		\prod_{i=1}^{\ell_U}
		\big(
		\dif^k\log \phi_i
		\big)^{\alpha_{k,i}}
		\prod_{i=\ell_U+1}^{n}
		\big(
		\dif^k\phi_i
		\big)
		^{\alpha_{k,i}}
		\bigg)\notag\\
		&=
		\underset{|\alpha_1|+2|\alpha_2|+\dots+k|\alpha_k|=m}
		{\sum_{\alpha_1,\dots,\alpha_k\in\mathbb{N}^n}}
		\chi_{\alpha_1,\dots,\alpha_k}
		\prod_{i=\ell_U+1}^{n}
		\big(
		\phi_i
		\big)^{\alpha_{1,i}}
		\dots
		\prod_{i=\ell_U+1}^{n}
		\big(
		\phi_i
		\big)^{\alpha_{k,i}}\times
		\notag\\
		&\times\bigg(
		\prod_{i=1}^{\ell_U}
		\big(
		\dif\log \phi_i
		\big)^{\alpha_{1,i}}
		\prod_{i=\ell_U+1}^{n}
		\bigg(
		\dfrac{\dif \phi_i}{\phi_i}
		\bigg)
		^{\alpha_{1,i}}
		\bigg)
		\dots
		\bigg(
		\prod_{i=1}^{\ell_U}
		\big(
		\dif^k\log \phi_i
		\big)^{\alpha_{k,i}}
		\prod_{i=\ell_U+1}^{n}
		\bigg(
		\dfrac{\dif^k\phi_i}{\phi_i}
		\bigg)
		^{\alpha_{k,i}}
		\bigg).
		\end{align}
	}
	Putting 
	\[
	\psi_{\alpha_1,\dots,\alpha_k}
	=
	\underset{|\alpha_1|+2|\alpha_2|+\dots+k|\alpha_k|=m}
	{\sum_{\alpha_1,\dots,\alpha_k\in\mathbb{N}^n}}
	\chi_{\alpha_1,\dots,\alpha_k}
	\prod_{i=\ell_U+1}^{n}
	\big(
	\phi_i
	\big)^{\alpha_{1,i}}
	\dots
	\prod_{i=\ell_U+1}^{n}
	\big(
	\phi_i
	\big)^{\alpha_{k,i}},
	\]
	and noting that the number of $U$ and $V$ are finite, one can conclude the proof.
\end{proof}

\subsection{Logarithmic directed manifolds and direct image formula}

A {\sl log--direct manifold} is a triple \[
(X,D,V),
\]
where $X$ is a projective manifold, $D$ is a simple normal crossing divisor on $X$ and $V$ is an analytic subbundle of the logarithmic tangent bundle $T_X(-\log D)$. Starting with a log--direct manifold $(X_0,D_0,V_0):=(X,D,T_X(-\log D))$, we then define $X_1:=\mathbb{P}(V_0)$ together with the natural projection $\pi_1:X_1\rightarrow X_0$. Setting $D_1:=\pi_1^*D_0$, so that $\pi_1$ becomes a log--morphism, $V_1\subset T_{X_1}(-\log D_1)$ is then defined by
\[
V_{1,(x,[v])}:=\{\xi\in T_{X_1,(x,[v])}(-\log D_1):\pi_*\xi\in \bF \cdot v\}.
\]
Any germ of a rigid analytic map $f:(\bF,0)\rightarrow (X\setminus D,x)$ can be lifted to $f^{[1]}: \bF\rightarrow X_1\setminus D_1$. By induction, we can construct on $X=X_0$ the {\sl Demailly-Semple tower}\,:
\[
(X_k,D_k,V_k)\rightarrow\dots \rightarrow(X_1,D_1,V_1)\rightarrow(X_0,D_0,V_0),
\] 
together with the projections $\pi_k:X_k\rightarrow X_0$. Denote by $\mathcal{O}_{X_k}(1)$ the tautological line bundle on $X_k$. Then the direct image $(\pi_k)_*\mathcal{O}_{X_k}(m)$ of $\mathcal{O}_{X_k}(m)=\mathcal{O}_{X_k}(1)^{\otimes m}$, denoted by $E_{k,m}T_X^*(\log D)$, is a locally free subsheaf of $E_{k,m}^{GG}T_X^*(\log D)$ generated by all polynomial operators in the derivatives up to order $k$, which are furthermore invariant under any change of parametrization $(\bF,0)\rightarrow (\bF,0)$. One can easily check the following
\begin{namedthm*}{Direct image formula}
	For any ample line bundle $\mathcal{A}$ on $X$, one has
	\begin{equation}
		\label{direct image formula}
	H^0\big(X,E_{k,m}T_X^*(\log D)\otimes\mathcal{A}^{-1}\big)\cong
	H_0\big(X_k,\mathcal{O}_{X_k}(m)\otimes\pi_k^*\mathcal{A}^{-1}\big).
	\end{equation}	
\end{namedthm*}

\subsection{Logarithmic Derivative Lemma for jet differentials}

\begin{thm}
	Let $X$ be a nonsingular projective variety defined over  $\bF$ and let $D$ be a simple normal crossing divisor on $X$. Let $\omega\in H^0(X,E^{GG}_{k,m}T_X^*(\log D))$ be a logarithmic jet differential of order $k$ and weight $m$ along $D$. Let $f$ be an analytic map from $A[r_1,r_2)$ to $X$. If 
	\[
	f^*\omega
	=\varphi(\xi)\big(\dif\xi\big)^m,
	\]
	where $\xi$ is the global coordinate of $\bF$, then
	\[
	|\varphi|_r
	\leq\dfrac{C}{r^m}\eqno\scriptstyle{(r_1\,\leq r\,<r_2)},
	\]
	for some constant $C$ which is independent of $r$.
\end{thm}

\begin{proof}
	Let $\mathcal{U}$ be an affinoid covering of $X$ as in Proposition~\ref{affinod covering and good express of jet}. Suppose that $f(\xi)\in U$ for some $U\in\mathcal{U}$. Then by~\eqref{local expression of jet differential via rational functions}, one has
	{\footnotesize
		\begin{align*}
		\varphi(\xi)
		&
		=
		\underset{|\alpha_1|+2|\alpha_2|+\dots+k|\alpha_k|=m}
		{\sum_{\alpha_1,\dots,\alpha_k\in\mathbb{N}^n}}
		\big(
		\psi_{\alpha_1,\dots,\alpha_k}
		\circ f
		\big)
		(\xi)
		\times\\
		&\times
		\bigg(
		\prod_{i=1}^{\ell_U}
		\big(
		\dif\log (\phi_i\circ f)(\xi)
		\big)^{\alpha_{1,i}}
		\prod_{i=\ell_U+1}^{n}
		\bigg(
		\dfrac{\dif (\phi_i\circ f)(\xi)}{(\phi_i\circ f)(\xi)}
		\bigg)
		^{\alpha_{1,i}}
		\bigg)
		\dots
		\bigg(
		\prod_{i=1}^{\ell_U}
		\big(
		\dif^k\log (\phi_i\circ f)(\xi)
		\big)^{\alpha_{k,i}}
		\prod_{i=\ell_U+1}^{n}
		\bigg(
		\dfrac{\dif^k(\phi_i\circ f)(\xi)}{(\phi_i\circ f)(\xi)}
		\bigg)
		^{\alpha_{k,i}}\bigg).
		\end{align*}
	}
	
	Hence
	
	{\footnotesize
		\begin{align*}
		|\varphi(\xi)|_r
		&
		\leq
		\underset{|\alpha_1|+2|\alpha_2|+\dots+k|\alpha_k|=m}
		{\max_{\alpha_1,\dots,\alpha_k\in\mathbb{N}^n}}
		\bigg\{
		|\big(
		\psi_{\alpha_1,\dots,\alpha_k}
		\circ f
		\big)
		(\xi)|_r
		\times\\
		&\times
		\left|
		\prod_{i=1}^{\ell_U}
		\big(
		\dif\log (\phi_i\circ f)(\xi)
		\big)^{\alpha_{1,i}}
		\prod_{i=\ell_U+1}^{n}
		\bigg(
		\dfrac{\dif (\phi_i\circ f)(\xi)}{(\phi_i\circ f)(\xi)}
		\bigg)
		^{\alpha_{1,i}}
		\bigg)
		\dots
		\bigg(
		\prod_{i=1}^{\ell_U}
		\big(
		\dif^k\log (\phi_i\circ f)(\xi)
		\big)^{\alpha_{k,i}}
		\prod_{i=\ell_U+1}^{n}
		\bigg(
		\dfrac{\dif^k(\phi_i\circ f)(\xi)}{(\phi_i\circ f)(\xi)}
		\bigg)
		^{\alpha_{k,i}}\right|_r
		\bigg\}.
		\end{align*}
	}
	Since $\psi_{\alpha_1,\dots,\alpha_k}$ are regular, and hence bounded, one can apply the Logarithmic Derivative Lemma to the remaining terms in the right hand side of the above inequality to obtain the desired estimate.
	
\end{proof}

\subsection{Proof of Theorem~\ref{big-Picard}}

	We mimic the reasoning of \cite[Theorem 6.1]{Cherry-Ru2004}.  Let $g:A[1/R,\infty)$  be the rigid analytic map obtained from $f$ by setting $g(\xi)=f(1/\xi)$.
	Since $\mathcal{A}$ is ample, one can take an integer $\ell$ large enough such that $\ell \mathcal{A}$ is very ample. By~Proposition~\ref{existence of very ample divisor meeting f often}, there exists another section $H$ in the linear system $|\ell\mathcal{A}|$ such that
	\[
	(1-\epsilon)T_g(r)\leq N_g(H,r)
	+
	O(\log r),
	\]
 and $\{g(\xi):f^*\omega(\xi)\not=0\}\not\subset H$. Let $\psi$ be a rational function on $X$ whose zero divisor is $H-\ell\mathcal{A}$. Then 
	\[
	\widetilde{\omega}:=\psi\,\omega^{\otimes \ell}
	\]
	is a logarithmic jet differential form of order $k$ and weight $\ell m$
	vanishing on $H$. Set
	\[
	g^*\widetilde{\omega}(\xi)
	=
	\varphi(\xi)(\dif\xi)^{\ell m}.
	\]
	Then $\varphi\not\equiv 0$ by the choice of $H$ and the assumption that $f^*\omega\not\equiv 0$. Furthermore, by the construction of $\widetilde{\omega}$, it is not hard to check that $N_g(H,r)\leq N_{\varphi}(0,r)$.
	Therefore
	\begin{align*}
	(1-\epsilon) T_g(H,r)
	&\leq N_g(H,r)+O(\log r)\\
	&\leq N_{\varphi}(0,r)
	+O(\log r)\\
	\text{[by the First Main Theorem]}\ \ \ \ \ &\leq \log|\varphi|_r
	+O(\log r).
	\end{align*}
	By the Logarithmic Derivative Lemma for logarithmic jet differential, one has $|\varphi|_r\rightarrow 0$ as $r\rightarrow\infty$. Hence, the above estimate yields $T_g(H,r)=O(\log r)$. We then apply Proposition~\ref{extension condition} to conclude.
\qed

\section{\bf Proofs of Theorems~\ref{thm-A},~\ref{thm-B}}
\label{Section 4}

\subsection{Transcendence of the uniformizing map}\label{non-arc_unif}

Let $R$ be a complete discrete valuation ring with fraction field $\bF$. Let $\pi$ be a generator of the maximal ideal of $R$ and $k = R/\pi R$ be the residue field.
Let $A$ be an abelian variety over $\bF$, with the  N\'eron model $\mathfrak{A}$  over $R$. Then $A \cong \mathfrak{A}_{\bF}$, the generic fiber of $\mathfrak{A}$. Let $\mathfrak{A}^0 \subset \mathfrak{A}$ be the identity component of the N\'eron model.
According to the Semi-stable Reduction Theorem {\cite[exp.IX~3.6]{SGA7v1}}, the closed fiber $\mathfrak{A}^0_k$ is a semi-abelian variety over $k$.
We say that $A$ is \emph{totally degenerate} if $\mathfrak{A}^0_k$ is a torus, i.e. $\mathfrak{A}^0_k \cong \mathbb{G}^g_{m,k}$, for $g = \mathrm{dim}\,A$.

We then consider the formal scheme $A^0_{\mathrm{for}}$ over the formal spectrum $\mathrm{Spf}\,R$, which is the formal completion of $\mathfrak{A}^0$ along the closed fiber $\mathfrak{A}^0_k$. On the other hand, we have the formal torus $\mathbb{G}^g_{m,\mathrm{for}}$ over $\mathrm{Spf}\,R$. Hence by the rigidity of tori {\cite[exp.IX~\S3]{SGA3II}}, the isomorphism between closed fibers can be lifted to
\[
\mathbb{G}^g_{m,\mathrm{for}} \to A^0_{\mathrm{for}}.
\]
Applying the functor of Raynaud generic fiber, we receive a rigid analytic map
\[
p:\, (\mathbb{G}^{\mathrm{an}}_{m,\bF})^g \to A^{\mathrm{an}}
\] 
between rigid analytic spaces over the field $\bF$, which is called the \emph{uniformizing map} of $A$.

Denote the split torus by $T = \mathbb{G}^g_{m,\bF}$, and the   uniformizing map by $p:\, T^{\mathrm{an}} \to A^{\mathrm{an}}$. Without risk of confusion, we also use $p$ to denote the group homomorphism
\[
p:\, T^{\mathrm{an}}(\bF) \to A^{\mathrm{an}}(\bF)
\]
between rigid analytic groups. Then the kernel $\Gamma:= \mathrm{Ker}(p) \subset T^{\mathrm{an}}(\bF)$ is a lattice. 

\begin{thm}[cf. Theorem~1.2 in \cite{BL91}]
	The quotient $T^{\mathrm{an}}/\Gamma$ is a rigid analytic space, and $p$ induces an isomorphism
	\[
	T^{\mathrm{an}}/\Gamma \xrightarrow{\cong} A^{\mathrm{an}}.
	\]  
\end{thm}

Now consider an affine curve $X \subset T$ with a projective completion $\bar{X}$. Denote by $p_X$ the  composition
\begin{align}\label{p_X}
X^{\mathrm{an}} \hookrightarrow T^{\mathrm{an}} \xrightarrow{p} A^{\mathrm{an}}.
\end{align}
Our goal is to show 

\begin{thm}\label{non-extendable}
	The rigid analytic map $p_X$ is not extendable, i.e. it can not be extended to a rigid analytic map $\bar{p}_X:\, \bar{X}^{\mathrm{an}} \to A^{\mathrm{an}}$ between projective varieties.
\end{thm}

A proof will be reached at the end of the next subsection.

\subsection{A central topological argument}
\label{area argument}
Let $X(T) := \mathrm{Hom}_K(T,\mathbb{G}_{m,\bF})$ be the character group of the torus $T$.
 By the group homomorphism
\[
\ell:\, T^{\mathrm{an}} \to \mathrm{Hom}(X(T),\mathbb{R}), \quad x \mapsto (\ell(x):\, \chi \mapsto - \log |\chi(x)|),
\]
and by fixing a basis $\{ \chi_1, \dots, \chi_g \}$ of $X(T)$, we receive that
   $X(T) \cong \mathbb{Z}^g$, and
 $\mathrm{Hom}(X(T),\mathbb{R}) \cong \mathbb{R}^g$. Consequently, for a lattice $\Gamma\subset T^{\mathrm{an}}(\bF)$, its image $\ell(\Gamma)$ is also a lattice in $\mathbb{R}^g$ in the usual sense (cf. {\cite[\S6.4]{FvdP}}).\

Now assume that the valuation of $R$ is discrete, and that $- \log |\pi| =1$. Then $-\log|\cdot|:\, \bF^* \to \mathbb{Z}$ yields
\[
\ell:\, T^{\mathrm{an}}(\bF) \to \mathrm{Hom}(X(T), \mathbb{Z}) \cong \mathbb{Z}^g.
\]
In particular, $\ell(\Gamma) \subset \mathbb{Z}^g$. Note that for any cube $S \subset \mathbb{R}^g$, the preimage $\ell^{-1}(S)$ is a rational subdomain of $T^{\mathrm{an}}$.

\begin{lem}\label{ring}
	For any point $x \in T^{\mathrm{an}}$, one can find a rational subdomain $R(x)$ containing $x$ such that for any $\gamma \in \Gamma \setminus \{1\}$, there holds
	$
	\gamma \, R(x) \cap R(x) = \varnothing
	$.
\end{lem}
\begin{proof}
	In coordinates $\ell(x)=(\ell(x)_i) \in \mathrm{Hom}(X(T),\mathbb{R}) \cong \mathbb{R}^g$. Hence for small $0<\epsilon\ll 1$, the $\epsilon$-cube neighborhood $S := \{ t=(t_i) \in \mathrm{Hom}(X(T),\mathbb{R})\,;\,\, |t_i- \ell(x)_i| \leqslant \varepsilon,\,\text{for all}\, 1\,\leq\,i\,\leq\, g \}$ satisfies that,
	 for any $a \in \mathbb{Z}^g \setminus \{0\}$, the intersection
	$
	(a + S) \cap S 
	$ is empty.
	Hence $R(x) := \ell^{-1}(S)$ satisfies our requirement.
\end{proof}

Since $\ell(\Gamma)$ is a lattice in $\mathbb{R}^g$, we can choose a fundamental domain $F \subset \mathbb{R}^g$  of it. Define $R_F := \ell^{-1}(F)$, then it is clear that $T^{\mathrm{an}} = \bigcup_{\gamma \in \Gamma}\gamma \, R_F$.

We now  prove Theorem~\ref{non-extendable} by contradiction. Suppose that $p_X$ is extendable. Then by the rigid GAGA theorem, it must be induced from an algebraic morphism
\[
\bar{p}_X:\, \bar{X} \to A
\]
between projective $\bF$-schemes. Denote by $W := \bar{p}_X(\bar{X})$ the image curve in $A$. Then $\bar{p}_X:\, \bar{X} \to W$ is a finite morphism between curves. Denote the degree of $\bar{p}_X:\, \bar{X} \to W$ by $N \in \mathbb{Z}_{>0}$.
Note that we can assume from the beginning that $\bar{X}$, hence also $W$, are connected.

For \emph{any} $w \in W$, one can find some point $x \in p^{-1}(w) \cap R_F$. Choosing a rational subdomain $R(x)$ containing $x$ as in Lemma~\ref{ring}.
Note that its image $p(R(x))$ is an admissible open subset of $A^{\mathrm{an}}$ {\cite[\S6.4]{FvdP}}. Now we consider the intersection $p(R(x)) \cap W^{\mathrm{an}}$. Although $W$ is connected, this intersection might have several connected components. However, we can show that
\begin{lem}
	One can find an admissible open subset $U(w) \subset A^{\mathrm{an}}$ containing $w$ such that 
	\begin{itemize}
		\item[1)] $U(w) \subset p(R(x))$;
		\item[2)] $U(w) \cap W^{\mathrm{an}}$ has only one connected component.
	\end{itemize}
\end{lem}
\begin{proof}
	Note that $p(R(x))$ is an admissible open set. Thus we can find an affinoid subdomain $U = \mathrm{Sp}\,B \subset p(R(x))$ which contains $w$. By the noetherian property of Tate algebras, we have 
	\[
	U \cap W^{\mathrm{an}} = \bigsqcup^s_{i =1} W_i
	\]
	where each $W_i$ is a connected component. One can assume that $w \in W_1$. Then we consider the analytic Zariski open subset $U \setminus \left( \bigcup^s_{i=2}W_i\right)$, which is an admissible open subset {\cite[\S2.3]{Con08}}. Then we define $U(w):= U \setminus \left( \bigcup^s_{i=2}W_i\right)$.
\end{proof}

Going back to the proof of Theorem~\ref{non-extendable}.
Denote by $W_{U(w)}:= W \cap U(w)$. Since $p:\, T^{\mathrm{an}} \to A^{\mathrm{an}}$ is a topological covering map, by shrinking $U(w)$, we may assume that the restriction $p|_{p^{-1}(U(w))}:\, p^{-1}U(w) \cap R(x) \to U(w)$ 
is an isomorphism. Therefore, we can find a lifting $\widetilde{W}_{U(w)}$ of $W_{U(w)}$ in $R(x)$. Observe that
\begin{itemize}
	\item[1)] $p^{-1}W_{U(w)} = \bigsqcup_{\gamma \in \Gamma} \gamma \, \widetilde{W}_{U(w)}$;
	\item[2)] $p^{-1}_XW_{U(w)} = X \cap p^{-1}W_{U(w)} = \bigsqcup_{\gamma \in \Gamma} (X \cap \gamma \, \widetilde{W}_{U(w)})$;
	\item[3)] $\# \pi_0(p^{-1}_XW_{U(w)}) = \# p^{-1}_X(w) \leq N$. 
\end{itemize}
Define
\[
\Sigma(w) := \{ \gamma \in \Gamma\,;\,\, X \cap \gamma \, \widetilde{W}_{U(w)} \neq \varnothing \}.
\]
Then observations 2) and 3) imply that 
\[
\# \Sigma(w) \leq N,
\]
for each $w \in W$.
By the compactness of $W^{\mathrm{an}}$, there exists a finite collection $\{U(w_i)\,;\,\, i =1,\dots,q\}$ with $W^{\mathrm{an}} \subset \bigcup^q_{i=1}U(w_i)$. 
\begin{lem}
	Fix an arbitrary $\gamma \in \Gamma$. Then $\gamma\,R_F \cap X^{\mathrm{an}} \neq \varnothing$ if and only if there exists some $w_i$ for $1 \leq i \leq q$ such that $X^{\mathrm{an}} \cap \gamma \, \widetilde{W}_{U(w_i)} \neq \varnothing$, or equivalently, $\gamma \in \bigcup^q_{i =1}\Sigma(w_i)$.
\end{lem}
\begin{proof} 
	($\Leftarrow$) is clear. Now we show ($\Rightarrow$).
	Suppose $\gamma\,R_F \cap X^{\mathrm{an}} \neq \varnothing$.  Then we can find some $w_i$ such that $p(\gamma\,R_F \cap X^{\mathrm{an}}) \cap U(w_i) = p(\gamma\,R_F \cap X^{\mathrm{an}}) \cap W_{U(w_i)}$ is nonempty. 
	Therefore the preimage $(\gamma\,R_F \cap X^{\mathrm{an}}) \cap p^{-1}W_{U(w_i)}$ is also nonempty, and thus
        \[
	(\gamma\,R_F \cap X^{\mathrm{an}}) \cap p^{-1}W_{U(w_i)} = \bigsqcup_{\gamma' \in \Gamma} (\gamma\,R_F \cap X^{\mathrm{an}}) \cap \gamma'\, \widetilde{W}_{U(w_i)} \neq \varnothing.
	\]
	This implies that $\gamma\,\widetilde{W}_{U(w_i)} \cap X^{\mathrm{an}} \neq \varnothing$.
\end{proof}

\begin{proof}[End of the proof of Theorem~\ref{non-extendable}]
By the above Lemma, $X^{\mathrm{an}}$ is contained in a finite union of translates of $R_F$. Thus we can find an affinoid subdomain $U \subset T^{\mathrm{an}}$ such that $X^{\mathrm{an}} \subset U$. Indeed, for instance, we can take a sufficiently large cube $S \subset \mathbb{R}^g$ and we choose $U$  to be the rational subdomain $\ell^{-1}(S)$. But this is impossible since $X^{\mathrm{an}}$ is the analytification of an algebraic curve {\cite[Proposition~2.8]{JV18}}.
\end{proof}

\subsection{Jets projection on abelian varieties and stabilizers of subvarieties}\label{jet-proj}

Denote by
\[
\pi_k:\, A_k \to A_0 = A
\]
the $k$-th projection in the Demailly-Semple tower associated to the abelian variety $A$ of dimension $g$. Then the flatness of the geometry of $A$ yields that
\[
A_k = A \times \mathcal{R}_{g,k},
\]
where $\mathcal{R}_{g,k}$ is a general fiber of $\pi_k$.

Now we consider a rigid analytic map $f:\, C^{\mathrm{an}} \to A^{\mathrm{an}}$ from a smooth quasi-projective curve $C$ to $A$. Let $Z$ be the Zariski closure of the image  $f(C^{\mathrm{an}})$ in $A$. Denote by 
\[
f_k:\, C^{\mathrm{an}} \to A^{\mathrm{an}}_k
\]
the $k$-th lifting of $f$. Let $Z_k \subset A_k$ be the Zariski closure of the image of $f_k$. The composition map
\[
\tau_k:\, Z_k \hookrightarrow A_k \twoheadrightarrow \mathcal{R}_{g,k}
\]
is called the {\it jets projection} on $Z_k$. 
There is only one alternative holds:
\begin{itemize}
\item for every $k \geqslant 1$, the general fiber of $\tau_k$ has positive dimension;
\item there exists a positive integer $k$ such that $\tau_k$ is generically finite onto its image.
\end{itemize}
For the first case, by the same argument as in {\cite[Proposition~5.3]{PS14}}, we have 
\begin{pro}\label{case-1}
Suppose that for each $k \geqslant 1$, the general fiber of $\tau_k$ has positive dimension. Then the dimension of the stabilizer
\[
\mathrm{Stab}(Z) := \{a \in A\,;\,\, a + Z =Z \}
\] 
of $Z$
is strictly positive.
\end{pro}

For the second case, we have
\begin{pro}\label{case-2}
Let $k$ be a positive integer such that the jet projection
\[
\tau_k:\, Z_k \to \mathcal{R}_{g,k}
\] 
is generically finite onto its image. Then there exists a jet differential $\omega \in H^0(Z, E_{k,m}T^*_Z \otimes \mathcal{A}^{-1})$
for some ample line bundle $\mathcal{A}$ over $Z$ and some positive integer $m$. 
\end{pro}
\begin{proof}
Denote by $\mathcal{O}(1)$ the tautological line bundle on $\mathcal{R}_{g,k}$. By  {\cite[page~31, Theorem 6.8]{Dem97}}, and by the
 Lefschetz principle,
 we know that $\mathcal{O}(1)$ is a big line bundle, and that  for every $m\geqslant 1$  the base locus of  $|\mathcal{O}(m)|$
avoids the {\it regular part}, defined by $1$-jet non-vanishing, of $\mathcal{R}_{g, k}$
. Since the jet projection $\tau_k$ is generically finite onto its image, which must have nonempty intersection with the {\it regular part},  $ \tau^*_k\mathcal{O}(1)
$
is a big line bundle on $Z_k$. Thus the statement follows from the Kodaira lemma and the direct image formula~\eqref{direct image formula}.
\end{proof}

\subsection{Non-archimedean Ax-Lindemann Theorem}
Our goal in this subsection is to prove~ Theorem~\ref{thm-A}. The proof goes first by treating the $1$-dimensional case. The general case then follows by some reductions.
We first recall the following terminology of~\cite{Nog18}.

\begin{defi}
Let $\gamma:\,X^{\mathrm{an}} \to A^{\mathrm{an}}$ be a rigid analytic map from a smooth quasi-projective variety $X$ to an abelian variety $A$. Then $\gamma$ is called \emph{strictly transcendental} if for every abelian subvariety $B \subsetneq A$ and for every projective compactification  $\bar{X}$ of $X$, the composition $X^{\mathrm{an}} \xrightarrow{\gamma} A^{\mathrm{an}} \twoheadrightarrow (A/B)^{\mathrm{an}}$ is either constant, or it can not be extended to $\bar{X}^{\mathrm{an}} \to (A/B)^{\mathrm{an}}$.  
\end{defi}

When the source $X$ is of $1$-dimension, such strictly transcendental maps enjoy the following property.

\begin{thm}\label{translate}
	Let $X$ be a smooth quasi-projective curve.
	Suppose that $\gamma:\, X^{\mathrm{an}} \to A^{\mathrm{an}}$ is a strictly transcendental map. Then the Zariski closure $\mathrm{Zar}(\gamma(X^{\mathrm{an}}))$ of the image of $\gamma$ is a translate of some abelian subvariety.  
\end{thm}

\begin{proof}
	Denote $Z := \mathrm{Zar}(\gamma(X^{\mathrm{an}}))$ the Zariski closure and $B := \mathrm{Stab}^0(Z)$ the identity component of the stabilizer. Consider the quotient map $A \twoheadrightarrow A/B$ and denote by $Z_{\mathrm{quot}}$ the image of $Z$ in $A/B$. For the case that $Z_{\mathrm{quot}}$ is zero-dimensional,  we have already proved that $Z$ is a finite union of translates of abelian subvariety. Thus we focus on the case that $\mathrm{dim}\,Z_{\mathrm{quot}}>0$. 
	
	By the structure theorem due to Ueno and Kawamata~\cite{Ueno75, Kaw80}, we know that $Z_{\mathrm{quot}}$ is of general type and its stabilizer has zero dimension. Hence the assumption of Proposition~\ref{case-1} does not happen here. Now thanks to Proposition~\ref{case-2}, we can find some nonzero jet differential $\omega \in H^0(Z_{\mathrm{quot}}, E_{k,m}T^*_{Z_{\mathrm{quot}}} \otimes \mathcal{A}^{-1})$
	for some ample line bundle $\mathcal{A}$ over $Z_{\mathrm{quot}}$ and some positive integer $m$. Denote by $f:\, X^{\mathrm{an}} \to Z^{\mathrm{an}}_{\mathrm{quot}}$ the induced analytic map. Since the image of $f_k$ in $Z_{\mathrm{quot},k}$ is Zariski dense (cf. Section~\ref{jet-proj}), we have $f^*\omega \not\equiv 0$. Thus by Theorem~\ref{big-Picard}, $f$ is extendable. This contradicts to the
	assumption of strict transcendence of $\gamma$.
	\end{proof}
Now we consider an affine curve $X \subset (\mathbb{G}_{m,\bF})^g$ contained in the uniformizing torus of $A$. Let $p_X$ be the restriction of the uniformizing map to $X$ as in \eqref{p_X}. By Theorem~\ref{translate}, to conclude Theorem~\ref{thm-A} in $1$-dimensional case, it suffices to verify the strictly transcendental property of $p_X$. Although we only need to consider the case that $X$ is a curve, we state the following theorem in its full generality.
\begin{thm}\label{strict-trans-p_X}
	Let $p:\, (\mathbb{G}^{\mathrm{an}}_{m,\bF})^g \to A^{\mathrm{an}}$ be the uniformizing map of a totally degenerate abelian variety $A$ over $\bF$. Let $X \subset (\mathbb{G}_{m, \bF})^g$ be an irreducible affine variety with the restricted map $p_X:= p|_{X^{\mathrm{an}}}$. Then $p_X$ is strictly transcendental.  
\end{thm}

\begin{proof}
	We argue by contradiction. Suppose that there exists an abelian subvariety $B \subsetneq A$ such that 
	\[
	X^{\mathrm{an}} \xrightarrow{p_X} A^{\mathrm{an}} \twoheadrightarrow (A/B)^{\mathrm{an}}
	\]
	is nonconstant and extendable. Then we can find an algebraic curve $C \subset X$ such that the restriction
	\[
	C^{\mathrm{an}} \xrightarrow{p_C} A^{\mathrm{an}} \twoheadrightarrow (A/B)^{\mathrm{an}}
	\]
	is also nonconstant and extendable.
	
	Since $A$ is totally degenerate, so is the abelian subvariety $B$. Precisely, for the identity component $\mathfrak{B}^0$ of the N\'{e}ron model $\mathfrak{B}$ of $B$, the special fiber $\mathfrak{B}^0_k$ is not only a split torus but also a subtorus of $\mathfrak{A}^0_k$. Therefore $p^{-1}(B^{\mathrm{an}})=:T^{\mathrm{an}}_B$ is a subtorus  of $(\mathbb{G}^{\mathrm{an}}_{m,\bF})^g$. Denote by $T_{\mathrm{quot}}$ the quotient torus. Then we have the following commutative diagram
	\[
	\xymatrix{
		C^{\mathrm{an}} \ar[r] & (\mathbb{G}^{\mathrm{an}}_{m,\bF})^g \ar@{->>}[r] \ar[d] & T^{\mathrm{an}}_{\mathrm{quot}} \ar[d] \\
		&  A^{\mathrm{an}}  \ar@{->>}[r] & (A/B)^{\mathrm{an}}.
	}
	\]
	By our assumption, the composition $C^{\mathrm{an}} \to (\mathbb{G}^{\mathrm{an}}_{m,\bF})^g \twoheadrightarrow T^{\mathrm{an}}_{\mathrm{quot}} \to (A/B)^{\mathrm{an}}$ is extendable.
However,  the image of following composition of algebraic morphisms
	\[
	C \hookrightarrow (\mathbb{G}_{m,\bF})^g \twoheadrightarrow T_{\mathrm{quot}}
	\]
	is an algebraic curve
	$C_{\mathrm{quot}}$ 
	in the torus $T_{\mathrm{quot}}$, thus by Theorem~\ref{non-extendable} the restricted uniformizing map
	\[
	C^{\mathrm{an}}_{\mathrm{quot}} \to T^{\mathrm{an}}_{\mathrm{quot}} \to (A/B)^{\mathrm{an}}
	\]
	is \emph{not} extendable, which is absurd.
\end{proof}

\begin{proof}[Proof of the Theorem~\ref{thm-A}]
The case that $\mathrm{dim}\,X=1$
	follows from Theorems~\ref{translate},~\ref{strict-trans-p_X}.

	Now we consider the case that $X$ is of higher dimension. We argue by contradiction.
	Suppose that $\mathrm{Zar}(p_X):= \mathrm{Zar}(p_X(X^{\mathrm{an}}))$ is not a finite union of translates of abelian subvariety. Then
	by the structure theorem of Ueno and Kawamata \cite{Ueno75, Kaw80}, we only need to consider the case that $\mathrm{Sp}(\mathrm{Zar}(p_X)) \subsetneq \mathrm{Zar}(p_X)$. In particular, we can assume that $Z:=\mathrm{Zar}(p_X)$ is of general type. Then one can find a point $z \in p(X^{\mathrm{an}}) \setminus \mathrm{Sp}(Z)^{\mathrm{an}}$. Now we choose a point $x \in p^{-1}_X(z)$ and an algebraic curve $C \subset X$ which passes through $x$. Since we already proved Theorem~\ref{thm-A} for curves, the Zariski closure $\mathrm{Zar}(p_X(C^{\mathrm{an}}))$ is a translate of abelian subvariety contained in $Z$. Thus $\mathrm{Zar}(p_X(C^{\mathrm{an}})) \subset \mathrm{Sp}(Z)$.
On the other hand, we also know that $z \in \mathrm{Zar}(p_X(C^{\mathrm{an}}))$ and $z$ is outside $\mathrm{Sp}(Z)$ by our choice of $z$, which gives us a contradiction.

\end{proof}

\subsection{Pseudo-Borel hyperbolicity of closed subvarieties of abelian varieties}

Let $X$ be a closed subvariety of an abelian variety $A$ over $\bF$. Here we do not assume that $A$ is totally degenerate. 
Recall that the \emph{special set} $\mathrm{Sp}(X)$ of $X$ is defined to be the union of translates of positive-dimensional abelian subvarieties of $A$ contained in $X$. Assume that $X$ is of general type, then by Kawamata's theorem {\cite[Theorem~4]{Kaw80}}, one has $\mathrm{Sp}(X) \subsetneq X$.

\begin{proof}[Proof of the Theorem~\ref{thm-B}]
	By {\cite[Theorem~3.1]{Sun-non-arc}} we can assume that $S$ is a quasi-projective curve. After replacing $X$ by the Zariski closure of $f(S^{\mathrm{an}})$ we can assume that $f$ has Zariski dense image in $X$. Note that this process will not violate the condition that $X$ is of general type since $f(S^{\mathrm{an}}) \not\subset \mathrm{Sp}(X)^{\mathrm{an}}$. Now we consider the jet projection
\[
\tau_k:\, X_k \hookrightarrow A_k \twoheadrightarrow \mathcal{R}_{g,k},
\]
as in Section~\ref{jet-proj}. Since $X$ is of general type,  the dimension of $\mathrm{Stab}(X)$ must be zero. Hence by Proposition~\ref{case-2}, we can find some nonzero jet differential $\omega \in H^0(X, E_{k,m}T^*_X \otimes \mathcal{A}^{-1})$
for some ample line bundle $\mathcal{A}$ over $X$ and some positive integer $m$. By the Zariski density assumption, we know that $f^*\omega \not\equiv 0$.  Now the algebraicity of $f$ follows from Theorem~\ref{big-Picard} and the rigid GAGA theorem.
\end{proof}

\begin{rmk}
Using Theorem~\ref{thm-B} and the criterion {\cite[Theorem~2.18]{JV18}}, it follows that a  general type  closed subvariety $X$ of an abelian variety over $\bF$ is $\bF$-analytically Brody hyperbolic modulo $\mathrm{Sp}(X)$. Note that this result was obtained in~\cite{Morrow2020}  by an alternative approach.

\end{rmk}

\begin{center}

\end{center}
\end{document}